\documentclass[12pt,a4paper]{amsart}

\usepackage[utf8]{inputenc}
\usepackage[T1]{fontenc}
\usepackage[english]{babel}
\usepackage{cite}

\usepackage{indentfirst}
\usepackage{amssymb}
\usepackage{amsfonts}
\usepackage{amsmath}
\usepackage{amsthm}
\usepackage[top=2.5cm, bottom=2.5cm, left=2.5cm, right=2.5cm]{geometry}
\usepackage{amsopn}
\usepackage[colorlinks]{hyperref}
\usepackage{float}
\usepackage{graphicx}
\usepackage{color}
\usepackage{xspace,colortbl}
\usepackage{siunitx}
\usepackage{bbm}

\newtheorem{theorem}{Theorem}[section]
\newtheorem{corollary}[theorem]{Corollary}
\newtheorem{lemma}[theorem]{Lemma}

\theoremstyle{definition}

\newtheorem{remark}[theorem]{Remark}

\DeclareMathOperator{\arginf}{arg\,inf}
\DeclareMathOperator{\bias}{bias}

\DeclareSymbolFont{bbsymbol}{U}{bbold}{m}{n}
\DeclareMathSymbol{\ind}{\mathbin}{bbsymbol}{'061}

\title[Parametric estimation of quantile versions of Zenga and D curves\ldots]{Parametric estimation of quantile versions of Zenga and D inequality curves: methodology and application to Weibull distribution}
\author[S{.} Pi\c{a}tek]{Sylwester Pi\c{a}tek}


\address[S.P.]{Faculty of Pure and Applied Mathematics\\ Wroc{\l}aw University 
	of Science and Technology\\
	Wybrze\.ze Wyspia\'nskiego 27,
	50-370 Wroc{\l}aw, Poland
}
\email{sylwester.piatek@pwr.edu.pl}

\begin{document}

\begin{abstract}
Inequality (concentration) curves such as Lorenz, Bonferroni, Zenga curves, as well as a~new inequality curve --- the $D$~curve, 
are broadly used to analyse inequalities in wealth and income distribution in certain populations. 
Quantile versions of these inequality curves are more robust to outliers. 
We discuss several parametric estimators of quantile versions of the Zenga and $D$~curves. 
A~minimum distance (MD) estimator is proposed for these two curves and the indices related to them. 
The consistency and asymptotic normality of the MD estimator is proved.
The MD estimator can also be used to estimate the inequality measures corresponding to the quantile versions of the inequality curves.
The estimation methods considered are illustrated in the case of the Weibull model, which is often applied to the precipitation data or times to the occurrence of a~certain event.
\end{abstract}

\maketitle

\section{Introduction}\label{sec:intro}
Inequality curves, for example the Lorenz curve \cite{Lorenz1905}, are broadly used to analyse inequalities of wealth and income distribution in certain populations (see e.g. \cite{Bandourian2002}). Dozens of other applications in various fields of science appeared in recent years, which leads to the idea of proposing some alternatives for the most popular inequality curve, i.e. the Lorenz curve, and for associated with it an inequality measure, which is the Gini index. Lorenz curve is not defined for distributions with infinite expected value and is sensitive to outliers. Some alternatives, like the Bonferroni curve \cite{Bonferroni1930}, the Zenga-07 curve \cite{Zenga2007} and $D$~curve \cite{Davydov2020}, are defined by their relationship with the Lorenz curve, and for this reason they share these drawbacks. 

Prendergast and Staudte proposed three quantile versions of the Lorenz curve that could also be applied to the distributions with infinite mean value \cite{Prendergast2016}. Some methods of estimating these curves were described by Siedlaczek \cite{Siedlaczek2018}. In \cite{Prendergast2018} Prendergast and Staudte also proposed the $R$ curve and elaborated on its properties, estimation, and possible applications. They also suggested an alternative curve, which together with the $R$ curve after a~minor transform gives quantile counterparts of the Zenga and $D$ curves, respectively, denoted here as the $qZ$ and $qD$ curves. 
The problem of nonparametric estimation of $qZ$ and $qD$ was covered by Jokiel-Rokita and Piątek in \cite{JokPia2023}.

In this paper we focus on the problem of parametric estimation of two quantile concentration curves, 
namely $qZ$, which is an alternative to the~Zenga curve \cite{Zenga2007} and $qD$, an alternative to the~$D$ curve \cite{Davydov2020}.
Let $X$ be a~nonnegative random variable with cumulative distribution function (cdf) $F_{\theta}(t)=P_{\theta}(X\leq t),$ where ${\theta}\in\Theta$ is an unknown parameter.
Denote by
\begin{equation*}
Q_{\theta}(p)=F_{\theta}^{-1}(p)=\inf\{t\in\mathbb{R}:F_{\theta}(t)\geq p\}
\end{equation*}
the quantile function (qf).
The $qZ$ and $qD$ concentration curves are defined as follows:
\begin{equation*}\label{e:qZ}
qZ(p;Q_{\theta})=1-\frac{Q_{\theta}(p/2)}{Q_{\theta}((1+p)/2)}, \qquad p\in(0,1)
\end{equation*}
and
\begin{equation*}\label{e:qD}
qD(p;Q_{\theta})=1-\frac{Q_{\theta}(p/2)}{Q_{\theta}(1-p/2)}, \qquad p\in(0,1).
\end{equation*}
In addition, these curves are defined at the extreme points $p=0$ and $p=1$ as follows:
\begin{align*}
    qZ(0;Q_{\theta})=1=qZ(1;Q_{\theta}), 
    \quad
    qD(0;Q_{\theta})=1
    \quad
    \text{and}
    \quad
    qD(1;Q_{\theta})=0.
\end{align*}
The perfect equality curve for $qZ$ and for $qD$ is equal to 0. An illustration of the example curves $qZ$ and $qD$ for the Weibull distribution with several values of the shape parameter is depicted in Figure \ref{fig:qz_qd_weibull_example}. The concentration indices $qZI(Q_{\theta})$ and $qDI(Q_{\theta})$ are defined as the area under curves $qZ$ and $qD$, respectively, that is
\begin{equation*}\label{e:qZI}
qZI(Q_{\theta})=\int_{0}^{1}qZ(p;Q_{\theta})\,dp,
\end{equation*}
\begin{equation*}\label{e:qDI}
qDI(Q_{\theta})=\int_{0}^{1}qD(p;Q_{\theta})\,dp.
\end{equation*}
Some applications of these indices to the salary data were presented in \cite{JokPia2023} and \cite{Brazauskas2023}.

\begin{figure}[H]
	\includegraphics[angle=0,width=\textwidth]
	{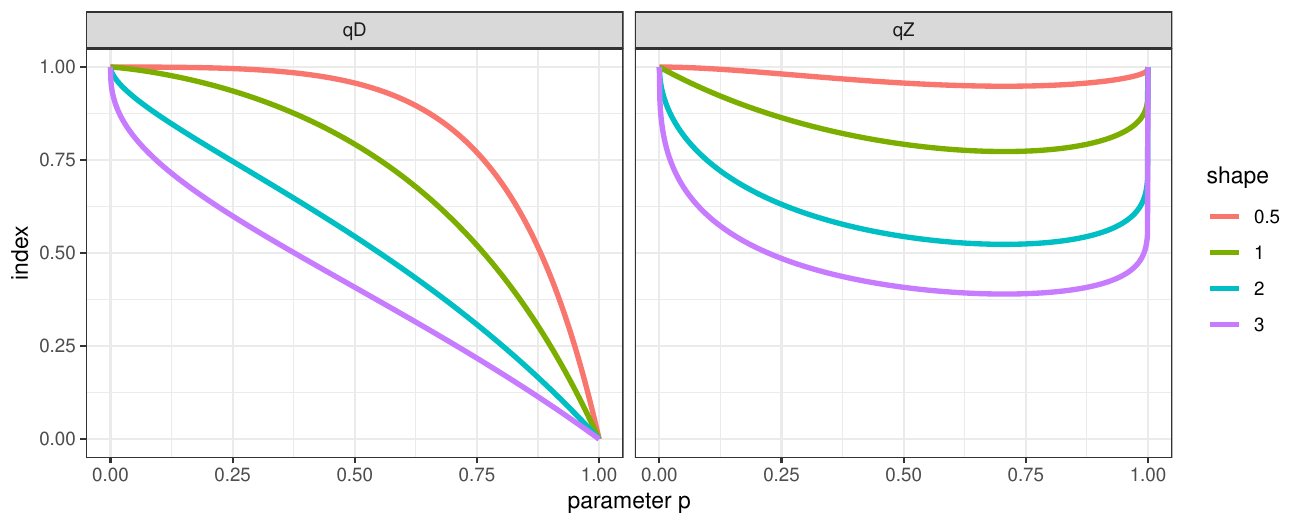}
	\caption[Comparison of the $qD$ (left plot) and $qZ$ (right plot) for Weibull distribution with different values of shape parameter]{Comparison of the $qD$ (left plot) and $qZ$ (right plot) for Weibull distribution with different values of shape parameter}
	\label{fig:qz_qd_weibull_example}
\end{figure}


The article is organised as follows.
Section \ref{sec:md_est} provides the definition of the proposed minimum distance estimator of the $qZ$ and $qD$ curves.
In Section \ref{sec:properties} the properties of the minimum distance estimators are proved. Section \ref{sec:weibull_distribution} introduces the Weibull distribution, which can be used as an example of a model for which $qZ$ and $qD$ can be estimated with MD estimators.
Section \ref{sec:simulation_study} provides a~comparison of several estimators of the $qZ$ and $qD$ curves in case of data from the Weibull distribution. 
In Section \ref{sec:real_data_analysis} an analysis of a~real data set with the concentration curves is given.
The last section contains a~brief summary of this paper.

\section{Minimum distance estimator}\label{sec:md_est}
One of the possible methods that can be used to estimate the $qZ$ and $qD$ curves is the minimum distance (MD) method. It was introduced by Wolfowitz (see for example \cite{Wolfowitz1957}) and has been widely used since then. 
More information on minimum distance estimation can be found in Basu et al. \cite{Basu2011} and references therein.
Millar described \cite{Millar1984} the minimum distance-type estimators of the parameter $\theta\in\Theta$ of a~distribution given with $F_{\theta}$ as follows. 
Denote by $\mathcal{H}$ the Hilbert space of the real bounded functions on $\mathbb{R}$. Let $F_{\theta},{F}_{n}\in \mathcal{H}$ and let $|\cdot|_{\mathcal{H}}$ be the norm of $\mathcal{H}$. A~minimum distance estimator of $\theta$ based on cdf is $\hat{\theta}$ such that
\begin{equation*}
    \hat{\theta}=\arginf_{\theta\in\Theta}|\xi_n(\theta)|_\mathcal{H},
\end{equation*}
where $\xi_n(\theta)$ is a~stochastic process with time parameter $\theta\in\Theta$ and values in $\mathcal{H}$. The process $\xi_n$ can be defined in various ways as a~difference between a~function of a~true parameter and its empirical counterpart. For example, $\xi_n$ could be
$|F_{\theta}-{F}_{n}|_\mathcal{H}$.
Since the goal here is to estimate the function $qZ(p;Q_{\theta})$, one can define
an MD estimator of $\theta$ based on the $qZ$ curve. This estimator $\hat{\theta}$ is found by minimising the distance between
$qZ_{\theta}=qZ(p;Q_{\theta})$ and the empirical 
curve denoted by
\begin{align}\label{eq:qzn}
    qZ_{n}=qZ(p;Q_{n}),
\end{align}
so it yields a~curve $qZ_{\hat{\theta}}=qZ(p;Q_{\hat{\theta}})$ that should best match the empirical curve. 
$Q_n$ stands here for empirical qf, given by
\begin{equation*}\label{e:Qn}
{Q}_{n}(p;{\bf X}):=\inf\{t:F_{n}(t;{\bf X})\geq p\},
\end{equation*}
where $F_n$ is the empirical cdf. 
We can define a~process
\begin{equation*}\label{eq:xi_n_e}
\xi_n^E(\theta) = qZ_{n}-qZ_{\theta}.
\end{equation*}
The formal definition of the MD estimator of the parameter $\theta$~obtained this way is as follows:
\begin{equation}\label{eq:mde_e_def}
    \hat{\theta}^E_n=\arginf_{\theta\in\Theta}|\xi_n^E(\theta)|_B.
\end{equation}

The most common choice of the norm of $\mathcal{H}$ in the problem of minimum distance estimation is $L_2$ norm, which was used, for example, to define MD estimators of the ODC curve \cite{Hsieh1996}, the ROC curve (see for example \cite{Davydov2009}, \cite{JokTop2019}, \cite{Gneiting2022} or \cite{JokTop2021}) or parameters of the $\alpha$-stable distribution (see \cite{Hopfner1999} or \cite{Fan2009}). 

Similarly, another MD estimator can be obtained basing on difference between $qZ_{\theta}$ and another nonparametric estimator of $qZ$. For example, a~continuous counterpart of $qZ_n$ can be considered.
Jokiel-Rokita and Piątek \cite{JokPia2023} suggested to use to estimate both $qZI$ and $qDI$ a~nonparametric plug-in estimator based on the Hyndman and Fan estimator (HF) \cite{Hyndman1996} of qf ($Q_n^{HF}$) when the finiteness of the expected value of the random variable is assumed. Weibull and Gumbel estimator (WG) \cite{Hyndman1996} of qf was recommended in the opposite case. These estimators are defined as a~linear interpolation between plotting positions $(p_k,X_{(k)})$, where
\begin{equation*}\label{ppHFandWG}
p_{k}^{HF}=\frac{k-1/3}{n+1/3}
\qquad
\text{and}
\qquad
p_{k}^{WG}=\frac{k}{n+1},
\end{equation*}
respectively for HF and WG estimators and $X_{(k)}$ is the $k$-th order statistic. 
In the intervals $[0,p_1]$ and $[p_n,1]$ such estimators have values $X_{(1)}$ and $X_{(n)}$ respectively. In consecutive sections only HF estimator is considered to maintain the clarity of the work, but WG estimator and some similar  continuous counterparts of empirical qf may be used.
Considering $qZ_{n}^{HF}$, defined by plugging-in the HF estimator of qf into the formula for $qZ$, we get
\begin{equation}\label{eq:mdhf_def}
    \hat{\theta}^{HF}=\arginf_{\theta\in\Theta}|\xi_n^{HF}(\theta)|_B,
\end{equation}
where
\begin{equation*}\label{eq:xi_n_hf}
\xi_n^{HF}(\theta)=qZ_{n}^{HF}-qZ_{\theta}.
\end{equation*}

\begin{remark}
    Analogous MD estimators can also be constructed for the $qD$ curve based on its empirical estimator $qD_n$ or the continuous estimator $qD_n^{HF}$. 
\end{remark}

\section{Properties of the estimators}\label{sec:properties}
In this section, the proofs of existence and asymptotic properties of MD estimators of $qZ$ (for a~given $p$) and $qZI$ (and $qD$ and $qDI$) are given. 
To provide the proofs we recall stochastic processes $\xi_n^E(\theta)$ and $\xi_n^{HF}(\theta)$ defined in Section \ref{sec:md_est}.
We denote by $T$ the continuous linear operator from $\Theta$ to $\mathcal{H}$, for which
\begin{equation}\label{eq:diff_cond}
    \xi_n^E(\theta)=\xi_n^E(\theta_0)+T(\theta-\theta_0)+o_P(\theta-\theta_0),
\end{equation}
\begin{equation*}\label{eq:diff_cond_hf}
    \xi_n^{HF}(\theta)=\xi_n^{HF}(\theta_0)+T(\theta-\theta_0)+o_P(\theta-\theta_0),
\end{equation*}
where $\theta_0$ is the true value of the parameter $\theta$. Using Taylor's formula we can show that
\begin{equation}\label{eq:t_def}
T(\theta)=\theta\frac{d}{d\theta_0}\left(qZ_{\theta_0}\right).
\end{equation}
Let 
\begin{equation}\label{e:def_eta}
\eta(\theta_0)=\frac{d}{d\theta_0}\left(qZ_{\theta_0}\right)
\end{equation}
and $\mathcal{B_{\eta}}$ be the linear space spanned by $\eta(\theta_0)$. Let $\pi$ be the orthogonal projection from $\mathcal{H}$ to $\mathcal{B_{\eta}}$. Now we introduce the following lemmas.
\begin{lemma}\label{qz_e_proc_conv}
If $F$ is differentiable at $F^{-1}(t)$, for a~given $t$,~with positive derivative
$Q^{'}(t)=f(F^{-1}(t))$, then
the process $\sqrt{n}\,\xi_n^E(\theta)$ converges in distribution to 
\begin{align*}
    G^{E}(t;\theta)=\left[1-qZ(t)\right]
    \left[ \frac{Q^{'}\left(\frac{t}{2}\right)}{Q\left(\frac{t}{2}\right)}B\left(\frac{t}{2}\right)
    -\frac{Q^{'}\left(\frac{1+t}{2}\right)}{Q\left(\frac{1+t}{2}\right)}B\left(\frac{1+t}{2}\right)\right],
\end{align*}
where $B$ is the standard Brownian bridge.
\end{lemma}
This lemma was proved as Theorem (5.7) in \cite{JokPia2023}.
\begin{lemma}\label{q_hf_est_conv}
Under the assumptions of Lemma \ref{qz_e_proc_conv}, we have
\begin{align*}
    \sup_{t\in[0,1]}|Q_n(t)-Q_n^{HF}(t)|=O_P\left(\frac{1}{n}\right).
\end{align*}
\end{lemma}
\begin{proof}
Let us notice that
\begin{align*}
    \sup_{t\in[0,1]}|Q_n(t)-Q_n^{HF}(t)|\leq
    \sup_{1\leq k,m< n}
    |p_k^E-p_k^{HF}||X_{(m+1)}-X_{(m)}|.
\end{align*}
The left factor is $O\left(\frac{1}{n}\right)$, since
\begin{align*}
\sup_{1\leq k\leq n}|p_k^E-p_k^{HF}|=
\sup_{1\leq k\leq n}\frac{\frac{1}{3}(k+n)}{n(n+\frac{1}{3})}=O_P\left(\frac{1}{n}\right).
\end{align*}
The latter is $O_P\left(1\right)$, because for any $p\in[m/n,(m+1)/n]$ we have
\begin{align*}
    \sup_{1\leq m< n}&|X_{(m+1)}-X{(m)}|\\
    &=\sup_{1\leq m< n, p\in[m/n,(m+1)/n]}|X_{(m+1)}-Q(p+1/n)
    +Q(p+1/n)-Q(p)+Q(p)-X{(m)}|\\
    &\leq
    \sup_{1\leq m< n, p\in[m/n,(m+1)/n]}|X_{(m+1)}-Q(p+1/n)|\\
    &+ \sup_{p\in[0,1]}|Q(p+1/n)-Q(p)|
    +\sup_{1\leq m< n, p\in[m/n,(m+1)/n]}|Q(p)-X{(m)}|.
\end{align*}
The first and third component of the above sum are $O_P(1)$, since the empirical quantile converges in probability to the true quantile. The middle one is $O_P\left(\frac{1}{n}\right)$, from the definition of the derivative of $Q$ in the point $p$. Thus, $\sup_{t\in[0,1]}|Q_n(t)-Q_n^{HF}(t)|$ is $O_P\left(\frac{1}{n}\right)$.
\end{proof}

\begin{lemma}\label{qz_hf_proc_conv}
Under the assumptions of Lemma \ref{qz_e_proc_conv}, we have
\begin{align*}
    \sup_{t\in[0,1]}|qZ_n(t)-qZ_n^{HF}(t)|
    =O_P\left(\frac{1}{n}\right).
\end{align*}
\end{lemma}
\begin{proof}
We have
\begin{align*}
    \sup_{t\in[0,1]}&|qZ_n(t)-qZ_n^{HF}(t)|\\
    &=
    \sup_{t\in[0,1]}
    \left|
    \frac{Q_n\left(\frac{t}{2}\right)}
    {Q_n\left(\frac{t+1}{2}\right)}
    -    \frac{Q_n^{HF}\left(\frac{t}{2}\right)}
    {Q_n\left(\frac{t+1}{2}\right)}
    +    \frac{Q_n^{HF}\left(\frac{t}{2}\right)}
    {Q_n\left(\frac{t+1}{2}\right)}
    -    \frac{Q_n^{HF}\left(\frac{t}{2}\right)}
    {Q_n^{HF}\left(\frac{t+1}{2}\right)}
    \right|\\
    &\leq \sup_{t\in[0,1]}
    \left| \frac{1}{Q_n\left(\frac{t+1}{2}\right)} \right|
    \left|Q_n\left(\frac{t}{2}\right)
    -Q_n^{HF}\left(\frac{t}{2}\right)\right|\\
    &\qquad+ \left| \frac{Q_n^{HF}\left(\frac{t}{2}\right)}{Q_n\left(\frac{t+1}{2}\right)Q_n^{HF}\left(\frac{t+1}{2}\right)} 
    \right|
    \sup_{t\in[0,1]}
    \left|Q_n\left(\frac{t+1}{2}\right)
    -Q_n^{HF}\left(\frac{t+1}{2}\right)\right|\\
    &\leq \frac{2}{Q_n(\frac{1}{2})} \sup_{t\in[0,1]}|Q_n(t)-Q_n^{HF}(t)|.
\end{align*}
Following Corrolary 21.5 in \cite{VanDerVaart1998} we have
\begin{align*}
Q_n(p)=Q(p)-\frac{1}{n}\sum_{i=1}^n\frac{\mathbbm{1}\{X_i\leq Q(p)\}-p}{f(Q(p))} + O_P(n^{-1/2}),
\end{align*}
thus $Q_n(p)=Q(p) + O_P(n^{-1/2})$, so the factor $2/Q_n(\frac{1}{2})$ is $O_P(1)$.
The second factor $\sup_{t\in[0,1]}|Q_n(t)-Q_n^{HF}(t)|$ is $O_P(\frac{1}{n})$, following Lemma \ref{q_hf_est_conv}.
Thus $\sup_{t\in[0,1]}|qZ_n(t)-qZ_n^{HF}(t)|$ is $O_P(\frac{1}{n})$.
\end{proof}

The asymptotic properties of the MD estimators $\xi_n^E(\theta)$ and $\xi_n^{HF}(\theta)$ are given in following theorem.

\begin{theorem}
Let us assume the identifiability, convergence, differentiability and boundedness of $\xi_n^E(\theta)$. We denote by $\hat{\theta}^E$ the MD estimator (\ref{eq:mde_e_def}). Then with probability approaching to 1 with $n\to\infty$, the estimator $\hat{\theta}^E$ exists and is unique. Moreover for $\hat{\theta}^E$ we have
\begin{equation*}\label{eq:main_thm_eq_1}
    \xi_n^E(\hat{\theta}^E_n)= (1-\pi)\circ \xi_n^E(\theta_0) +o_P(n^{-1/2}),
\end{equation*}
\begin{equation*}\label{eq:main_thm_eq_2}
    \hat{\theta}^E_n-\theta_0 = -T^{-1}\circ \pi\circ \xi_n^E(\theta_0) +o_P(n^{-1/2}),
\end{equation*}
\begin{equation*}\label{eq:main_thm_eq_3}
    \sqrt{n}\left(\xi_n^E(\hat{\theta}^E_n) - \xi_n^E(\theta_0)\right)\xrightarrow{d} \pi\circ G^E(t;\theta_0)\,\, \text{in}\, \mathcal{H},
\end{equation*}
\begin{equation}\label{eq:main_thm_eq_4}
    \sqrt{n}\left(\hat{\theta}^E_n-\theta_0\right)
    \xrightarrow{d}-T^{-1}\circ \pi\circ G^E(t;\theta_0)\,\, \text{in}\, \mathbb{R}^d.
\end{equation}

\end{theorem}

\begin{proof}
The proof is based on Theorem (3.6) from Millar's work \cite{Millar1984}.
In order to use this theorem, we need to check the four assumptions about the process $\xi_n^E(\theta)$. 
The first one, identifiability, is fulfilled since $\xi_n^E(\theta)-\xi_n^E(\theta_0)=qZ_{\theta}(t)-qZ_{\theta_0}(t)$ which is not random and does not depend on $n$. The convergence was shown in Lemma \ref{qz_e_proc_conv}. The differentiability is given by the existence of $T$ defined as in (\ref{eq:t_def}) which satisfies (\ref{eq:diff_cond}). Finally, $\xi_n^E(\theta)$ is obviously bounded. All the assumptions are satisfied, thus the theorem is proved. 
\end{proof}

\begin{remark}
Following Lemma \ref{qz_hf_proc_conv} the difference between $qZ_n$ and $qZ_n^{HF}$ is of $O_P\left(\frac{1}{n}\right)$, thus
all these properties also hold for $\hat{\theta}^{HF}$ defined in (\ref{eq:mdhf_def}).
\end{remark}

\begin{corollary}
The MD estimator $\hat{\theta}_n$ from the previous theorem is asymptotically normal with mean $0$ and variance given in the proof.
\end{corollary}
\begin{proof}
As we can see in equation (\ref{eq:main_thm_eq_4}), the asymptotic distribution of $\hat{\theta}$ can also be obtained. Following Millar (see equation (2.20) in \cite{Millar1984}) we have the variance of the MD estimator $\hat{\theta}_n^E$ given by 
\begin{equation*}\label{eq:var_thety}
\sigma_E^2=C^{-1}AC, 
\end{equation*}
where $C$ and $A$ are matrices with entries
\begin{equation*}\label{eq:def_c_and_a_dd}
    C_{ij}=\int_0^1 \eta_i(t) \eta_j(t) \,dt, \qquad A_{ij}=\int_0^1\int_0^1\eta_i(s)\eta_j(t)R(s,t)\,ds\,dt,
\end{equation*}
which reduces in one-dimensional case to
\begin{equation*}\label{eq:def_c_and_a_1d}
    \sigma_E^2=A=\int_0^1\int_0^1\eta(s)\eta(t)R(s,t)\,ds\,dt.
\end{equation*}
The function $\eta$ is defined as in (\ref{e:def_eta}) and $R(s, t)$ denotes here the~covariance function of $G^E(t;\theta)$, which is equal to

\begin{align*}\label{eq:R_s_t}
R(s,t)& = a(t)a(s)\left( \frac{1}{2}\min\{t,s\}-\frac{ts}{4}\right)
-b(t)a(s)\left(\frac{s}{2}-\frac{(1+t)s}{4} \right) \\
&+b(t)b(s)\left( \frac{1}{2}(1+\min\{t,s\})-\frac{(1+t)(1+s)}{4} \right)
-a(t)b(s)\left(\frac{t}{2}-\frac{(1+s)t}{4} \right),
\end{align*}
where
\begin{equation*}\label{eq:ab_def}
a(t) = \left[1-qZ(t)\right] \frac{Q^{'}\left(\frac{t}{2}\right)}{Q\left(\frac{t}{2}\right)}
\qquad 
\text{and}
\qquad 
b(t) = \left[1-qZ(t)\right] \frac{Q^{'}\left(\frac{1+t}{2}\right)}{Q\left(\frac{1+t}{2}\right)}.
\end{equation*}
\end{proof}

Denote the plug-in estimators of $qZ$ curve $\widehat{qZ}^{E}(p;{\bf X})=qZ(p;Q_{\hat{\theta}^{E}})$ and $\widehat{qZ}^{HF}(p;{\bf X})=qZ(p;Q_{\hat{\theta}^{HF}})$ and plug-in estimators of $qZI$.

\begin{corollary}
    The estimators $\widehat{qZ}^{E}(p;{\bf X})$ and $\widehat{qZ}^{HF}(p;{\bf X})$ are consistent estimators of $qZ(p;Q)$ for every $p\in[0,1]$. 
\end{corollary}

\begin{proof}
    This is a consequence of applying the delta method (see for example Chapter 3 in \cite{VanDerVaart1998}), which can be done, since $qZ(p;Q_{\hat{\theta}})$ is a~continuous function of $\theta$ for each $p\in[0,1]$.
\end{proof}

\begin{corollary}
    The estimators $\widehat{qZI}^{E}({\bf X})$ and $\widehat{qZI}^{HF}({\bf X})$ are consistent estimators of $qZI(Q)$. 
\end{corollary}

\begin{proof}
From the definition of $\widehat{qZI}^{E}({\bf X})$ and $qZI(Q),$ we have
\begin{equation*}
\widehat{qZI}^{E}({\bf X})-qZI(Q)=\int_{0}^{1}
\left[\widehat{qZ}^{E}(p;{\bf X})-qZ(p;Q)\right]\,dp.
\end{equation*}
We have though
\begin{equation*}
\int_{0}^{1}\left[\widehat{qZ}^{E}(p;{\bf X})-qZ(p;Q)\right]\,dp
\leq
\sup_{p\in(0,1)}|\widehat{qZ}^{E}(p;{\bf X})-qZ(p;Q)|.
\end{equation*}
The right-hand side of the inequality above converges to 0 when $n$ tends to infinity, since the estimator $\widehat{qZ}^{E}(p;{\bf X})$ is a~consistent estimator of $qZ$ for each $p \in [0,1]$. Thus $\widehat{qZI}^{E}({\bf X})$ converges in probability to $qZI(Q)$.
The proof for $\widehat{qZ}^{HF}$ follows the same pattern.
\end{proof}

\begin{remark}
All the properties shown in this section for $qZ$ also hold for $qD$ and can be proved analogously.
\end{remark}

\begin{remark}
 MD estimators of parameter $\theta$ based on empirical $qZ$ can also be used to construct a plug-in estimator of $qD$ and vice versa. 
\end{remark}

\section{Weibull distribution} \label{sec:weibull_distribution}
The two-parameter Weibull distribution $\mathcal{W}(\beta,\sigma)$ is defined by
the probability density function
\begin{align*}
f(x) = \frac{\beta}{\sigma}\left(\frac{x}{\sigma}\right)^{\beta-1}
\exp\Big[-\left(\frac{x}{\sigma}\right)^{\beta}\Big], \quad x\geq0, \beta>0, \sigma>0,
\end{align*}
where $\sigma$ is a~scale parameter and $\beta$ is a~shape parameter \cite{Weibull1939}.
\subsection{Concentration curves and measures for the Weibull distribution}
This distribution is utilised for example to fit precipitation data \cite{Chen2022}, travel times \cite{Al-Deek2007} or time intervals of the first passage process of foreign exchange rates \cite{Sazuka2007, Chakraborti2015}. The Gini index sometimes appears as a~measure used to analyse inequalities in such data \cite{Duan2021, Sazuka2007, Lee2019}.
Lorenz curve is also applied in reliability problems \cite{Chandra1981}, where the Weibull distribution often appears.

The qf of the Weibull distribution is given by 
\begin{align*}
Q(p)=\sigma\left[-\log(1-p)\right]
^{1/\beta}.
\end{align*}
Thus, the curves $qZ$ and $qD$ for the Weibull distribution are
\begin{equation*}\label{e:qZ_Weibull}
qZ(p;\beta)=1-\left[\frac{\log(1-p/2)}{\log((1-p)/2)}\right]^
{1/\beta}
\end{equation*}
and
\begin{equation*}\label{e:qD_weibull}
qD(p;\beta)=1-\left[\frac{\log(1-p/2)}{\log(p/2)}\right]
^{1/\beta},
\end{equation*}
respectively, hence they are scale invariant and depend only on the shape parameter $\beta$.

Figure \ref{fig:three_indices} depicts the relation between the shape parameter of the Weibull distribution and the values of three indices, namely the already mentioned $qZI$ and $qDI$, as well as the Gini index ($GI$). It can be deduced from this plot that $GI$ is more volatile then other two indices when the shape parameter $\beta$ is small (smaller than 2), however $qZI$ and $qDI$ might tend to be more susceptible (than $GI$) to the changes of $\beta$ when 
it is greater than 2.

\begin{figure}[H]
	\includegraphics[angle=0,width=\textwidth]
	{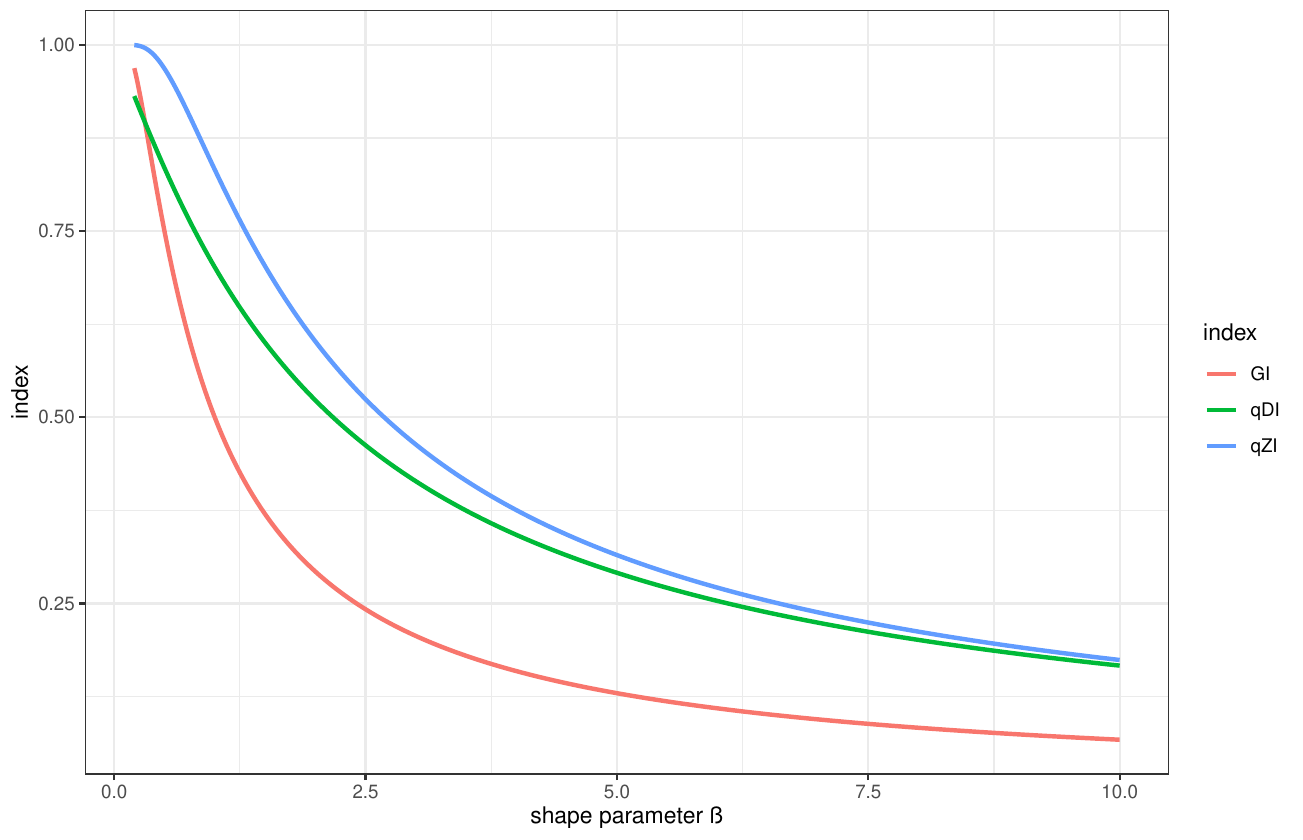}
	\caption[Comparison of the $GI$, $qZI$ and $qDI$ for different values of shape parameter of the Weibull distribution]{Comparison of the $GI$, $qZI$ and $qDI$ for different values of shape parameter of Weibull distribution}
	\label{fig:three_indices}
\end{figure}

\subsection{Plug-in estimators based on estimators of the shape parameter}\label{ml_est}
A plug-in estimators of $qZ$ and $qZI$, having $\mathbf{X} = (X_1, \ldots , X_n)$ which is a~random sample from Weibull distribution, are defined by
\begin{equation*}\label{e:qZ_Weibull_hat}
\widehat{qZ}(p;{\bf X})=1-\left[\frac{\log(1-p/2)}{\log(1/2-p/2)}\right]^
{1/\hat{\beta}}
\end{equation*}
and
\begin{equation*}\label{e:qZI_Weibull_hat}
\widehat{qZI}({\bf X})=\int_{0}^{1}\widehat{qZ}(p;{\bf X})\,dp,
\end{equation*}
where $\hat{\beta}$ is an estimator of the shape parameter $\beta$ of the Weibull distribution.

Although the MD estimators of the parameter $\beta$ and aforementioned concentration curves and measures have desired properties, some standard estimators of $\beta$ might lead to plug-in estimators of curves and measures having lower MSE and MISE. 
The popular methods of estimation of parameter $\beta$ were compared in many works in terms of their bias and MSE, but there are not much papers about the problem of estimation of curves depending on the parameters of the Weibull distribution, which can be estimated with a~plug-in estimator.
Thus in the simulation study in Section \ref{sec:simulation_study} we will check MSE and MISE of plug-in estimators of some functions of parameter $\beta$, namely $qZ$, $qD$, $qZI$ and $qDI$.

Among over a~dozen of methods applied to the problem of estimating $qZ$ and $qD$ and the corresponding indices, we decided to present here the result for three of them, further denoted by ML, MML, and BCML.
Results, including other methods, are available in the supplementary materials given in Section \ref{sec:supplement}.

\subsubsection{ML estimator}
For $X_{1},\ldots,X_{n}$
from the Weibull distribution,
the maximum likelihood (ML) estimator (see for example section 3.1 in \cite{Gebizlioglu2011}) $\beta^{ML}$ of $\beta$ 
is a~solution to the equation
\begin{equation}\label{e:betaMLEeq}
\sum_{i=1}^{n}X_{i}^{\beta}
+\frac{1}{n}\Big[\sum_{i=1}^{n}\log(X_{i}^{\beta})\Big]\sum_{i=1}^{n}X_{i}^{\beta}
-\sum_{i=1}^{n}\Big[X_{i}^{\beta}\log(X_{i}^{\beta})\Big] = 0.
\end{equation}
Maximum likelihood estimator for the Weibull distribution in asymptotically normal (see for example Section 5.5 in \cite{VanDerVaart1998}).
For each $p\in[0,1]$, the plug-in estimator $\widehat{qZ}^{ML}(p;{\bf X})$ based on $\hat{\beta}^{ML}$ is asymptotically normal (thus also consistent). This is again due to the delta method. Analogously, the plug-in estimator $\widehat{qZI}^{ML}(p;{\bf X})$ is asymptotically normal and consistent, since $qZI$ is a~continuous function of $\beta$.
\subsubsection{Modified ML estimator (MML)}
Many modifications of the ML estimator of the parameters of Weibull distribution were described in the literature.
Here, we use a~modification obtained recently by replacing equation (\ref{e:betaMLEeq}) with an unbiased estimating equation (see Section 3.1 in \cite{JokPia2022}). This leads to the estimator $\hat{\beta}^{MML}$ with a~bias significantly smaller than $\hat{\beta}^{ML}$. The asymptotic normality of $\hat{\beta}^{MML}$ can be shown analogously to that of $\hat{\beta}^{ML}$ , thus the asymptotic normality holds also for $\widehat{qZ}^{MML}(p;{\bf X})$, for a~given $p$, and $\widehat{qZI}^{MML}({\bf X})$.
\subsubsection{Bias-corrected ML estimator (BCML)}
Makalic and Schmidt \cite{Makalic2023} proposed in 2023 a~modification of the ML estimator of the shape parameter motivated by the need to reduce its bias. They proved that the bias of the regular ML estimator of the shape parameter is as follows,
\begin{align*}
    \bias(\hat{\beta}^{ML})=\beta\left(
    \frac{18(\pi^2-2\zeta(3))}{n\pi^4}\right) + O(n^{-2})
    \sim\beta\left(\frac{1.3795}{n}\right),
\end{align*}
where $n$ is the size of a~sample and $\zeta$ is the Riemann zeta function.
Their estimator is constructed simply by subtracting from MLE its bias, namely 
\begin{align*}
\hat{\beta}^{BCML}=
\hat{\beta}^{ML}\left(1-\frac{1.3795}{n}\right).
\end{align*}
Obviously, it has a significantly smaller bias and MSE than regular ML. Its performance is comparable to that of two other modified ML estimators described in that paper. 
Although the simulation study presented by the authors shows that an MLP method proposed by Shen and Yang \cite{Shen2015} is slightly better than BCML for complete samples, here we discuss only BCML, since MLP requires bootstrapping, which increases the complexity of the estimation process.
Since $\hat{\beta}^{BCML}$ is a $\hat{\beta}^{ML}$ shifted by an expression converging to 0 (when $n$ tends to $\infty$), it preserves the asymptotic normality of $\hat{\beta}^{ML}$, thus $\widehat{qZ}^{BCML}(p;{\bf X})$ and $\widehat{qZI}^{BCML}({\bf X})$ are asymptotically normal as well.
\subsubsection{Other estimators} In a~recent article by Jokiel-Rokita and Piątek \cite{JokPia2022} several estimators of the shape parameter are described and used to define plug-in estimators of extreme quantiles (e.g. $q=.99)$. All of these estimators can be used to define plug-in estimators of $qZ$, $qD$, $qZI$, and $qDI$. The performance of estimators obtained in such a way was compared with that of the aforementioned plug-in estimators. 
Section \ref{sec:simulation_study} contains only the most relevant results. Full comparison, including 11 plug-in estimators 
is available in the suplementary materials. For more insight on various estimators of parameters of the Weibull distribution, see for example \cite{Gebizlioglu2011}, \cite{Teimouri2011}, \cite{JokPia2022}, \cite{Yang2022}, \cite{Kim2023} or \cite{Makalic2023} and references therein.

\subsection{Starting point for MD estimators}
Since the MD method searches for a~value which minimises some expression, it requires a~good starting point to ensure a quick convergence to the minimum. Estimators of shape parameter (ML, MML and BCML), mentioned in the previous section, can be a~good starting points, as they have low bias and MSE. Their drawback is that they also require some time to be found, since they are defined as a~solution to an equation. For this reason, in order to save time, an estimator of the shape parameter given with a~closed-form expression can be used as a~starting point. An example of such estimator, based on quantiles of order .31 and .63, was presented by Seki and Yokoyama (see Section 2 in \cite{Seki1993}). We denote this estimator as PE.
Another one is the L-Moment estimator (LM), see for example \cite{Teimouri2011}. 
Two other estimators, based on nonparametric estimators of Gini index and denoted by G1 and G2, were proposed by Jokiel-Rokita and Piątek (see Section 3.2 in \cite{JokPia2022}), one of them being equal to the LM estimator. Another such estimator (TMML), a~modification of ML, was proposed by Gebizlioglu et al. (see  Section 3.5 in \cite{Gebizlioglu2011}). 

The advantage of PE, LM, G1, G2, and TMML, over the ML and its modifications is that they can also be applied to data containing zeros. The additional advantage of TMML is also low bias and MSE since it gives values almost equal to those of ML and MML.

\section{Simulation study}\label{sec:simulation_study}
Estimators of the curves $qZ$ and $qD$ for the Weibull distribution will be compared in this section, namely three plug-in estimators (using $\hat{\beta}^{ML}$, $\hat{\beta}^{MML}$ and $\hat{\beta}^{BCML}$), two MD estimators based on empirical (E) and HF estimators of $qZ$ (or $qD$) and a~nonparametric estimator (HF). 
The performance of the estimators presented in the previous section was compared in a~simulation study. Four values of the shape parameter $\beta$ were considered, namely $\beta=0.5$, resulting in a~heavy-tailed distribution, and $\beta\in\{1,2,3\}$, resulting in lighter tails.
For each case, 10,000 samples were produced. $L^2$ norm was used as a~criterion to assess the estimators of the curves. In the following figures and tables, the abbreviation HF stands for a~nonparametric plug-in estimator defined in (\ref{eq:qzn}), MDE and MDHF are plug-in estimators based on MD estimators $\hat{\beta}^{E}$ and $\hat{\beta}^{HF}$ described in Section \ref{sec:md_est}, while ML, MML, BCML stand for plug-in estimators based on estimators $\hat{\beta}^{ML}$, $\hat{\beta}^{MML}$ and $\hat{\beta}^{BCML}$  given in Section \ref{ml_est}.
The experiment was carried out using the R programming language, including \textit{optim} function for finding minimum of the expressions and \textit{integrate} function for computing integrals.

\subsection{Curves estimation}
We use 
mean integrated square error (MISE), defined for $qZ$ as
\begin{align*}
    \text{MISE}(\hat{\beta})=\mathbb{E}
    \left(
    \int_0^1 
    \left(qZ(p;\hat{\beta})-qZ(p;\beta)\right)^2\,dp
    \right),
\end{align*}
to assess the performance of the estimators.  
Tables \ref{tab:mise_qz} and \ref{tab:mise_qd} show the mean values of the error. There, it can be observed that MML and BCML are slightly better than ML in all cases and also that the difference between the errors of MD-based and ML-based methods is much smaller for larger $\beta$. The BCML estimator has the lowest MSE for $\beta\in\{0.5,1,2\}$ and ML for $\beta=3$ for both curves and both sample sizes considered in these tables. 

\begin{table}
\centering
\caption{MISE (multiplied by 1,000) of the estimators of $qZ$ for sample sizes $n=30$ and $n=100$ and shape parameter varying between $0.5$ and $3$}
\begin{tabular}{c|cccc|cccc}
  & \multicolumn{4}{c|}{$30$} & \multicolumn{4}{c}{$100$} \\
\cline{2-9}
 & {0.5} & {1} & {2} & {3} & {0.5} & {1} & {2} & {3} \\
\hline
  HF & 0.861 & 4.399 & 7.155 & 6.705 & 0.215 & 1.358 & 2.204 & 2.086 \\ 
  MDE & 0.735 & 2.851 & 3.633 & 2.946 & 0.147 & 0.764 & 1.031 & 0.852 \\ 
  MDHF & 0.635 & 2.596 & 3.429 & 2.822 & 0.140 & 0.735 & 1.005 & 0.832 \\ 
  ML & 0.435 & 2.152 & 2.912 & 2.274 & 0.093 & 0.554 & 0.811 & 0.659 \\ 
  MML & 0.375 & 1.968 & 2.759 & \color{red}2.216 & 0.088 & 0.530 & 0.792 & \color{red}0.654 \\ 
  BCML & \color{red}0.321 & \color{red}1.857 & \color{red}2.745 & 2.268 & \color{red}0.084 & \color{red}0.519 & \color{red}0.791 & 0.660 \\ 
\end{tabular}
\label{tab:mise_qz}
\end{table}

\subsection{Indices estimation}
Figure \ref{fig:index_boxplot} shows boxplots of the estimators of the indices. Similarly as in the case of curve estimation, MML and BCML outperform other estimators, especially in the case of heavier tails.  
Tables \ref{tab:mse_qzi} and \ref{tab:mse_qdi} compare the estimators of indices in terms of their MSE. 
We can also notice that MML and BCML have lower MSE than ML in all presented cases. 
BCML performs best for smaller values of $\beta$~and MML for larger ones.

\begin{table}
\centering
\caption{MISE (multiplied by 1,000) of the estimators of $qD$ for sample sizes $n=30$ and $n=100$ and shape parameter varying between $0.5$ and $3$}
\begin{tabular}{c|cccc|cccc}
  & \multicolumn{4}{c|}{$30$} & \multicolumn{4}{c}{$100$} \\
\cline{2-9}
 & {0.5} & {1} & {2} & {3} & {0.5} & {1} & {2} & {3} \\
\hline
  HF & 5.914 & 5.483 & 4.834 & 4.261 & 2.033 & 1.853 & 1.534 & 1.320 \\ 
  MDE & 4.192 & 3.675 & 3.018 & 2.408 & 1.286 & 1.108 & 0.882 & 0.707 \\ 
  MDHF & 4.037 & 3.546 & 2.900 & 2.325 & 1.269 & 1.099 & 0.867 & 0.696 \\ 
  ML & 0.989 & 1.604 & 1.982 & 1.732 & 0.259 & 0.428 & 0.544 & 0.495 \\ 
  MML & 0.918 & 1.492 & 1.865 & \color{red}1.671 & 0.253 & 0.408 & 0.530 & \color{red}0.489 \\ 
  BCML & \color{red}0.880 & \color{red}1.450 & \color{red}1.843 & 1.688 & \color{red}0.251 & \color{red}0.404 & \color{red}0.529 & 0.492 \\  
\end{tabular}
\label{tab:mise_qd}
\end{table}

\begin{table}
\centering
\caption{MSE (multiplied by 1,000) of the estimators of $qZI$ for sample sizes $n=30$ and $n=100$ and shape parameter varying between $0.5$ and $3$}
\begin{tabular}{c|cccc|cccc}
  & \multicolumn{4}{c|}{$30$} & \multicolumn{4}{c}{$100$} \\
\cline{2-9}
 & {0.5} & {1} & {2} & {3} & {0.5} & {1} & {2} & {3} \\
\hline
  HF & 0.540 & 2.444 & 3.362 & 2.727 & 0.112 & 0.675 & 0.992 & 0.818 \\ 
  MDE & 0.631 & 2.714 & 3.613 & 2.927 & 0.123 & 0.722 & 1.025 & 0.847 \\ 
  MDHF & 0.544 & 2.466 & 3.408 & 2.805 & 0.117 & 0.694 & 0.998 & 0.827 \\ 
  ML & 0.371 & 2.046 & 2.896 & 2.260 & 0.077 & 0.523 & 0.806 & 0.656 \\ 
  MML & 0.318 & 1.866 & 2.742 & \color{red}2.203 & 0.073 & 0.500 & 0.787 & \color{red}0.650 \\ 
  BCML & \color{red}0.270 & \color{red}1.753 & \color{red}2.725 & 2.257 & \color{red}0.070 & \color{red}0.489 & \color{red}0.786 & 0.656 \\  
\end{tabular}
\label{tab:mse_qzi}
\end{table}

\begin{table}
\centering
\caption{MSE (multiplied by 1,000) of the estimators of $qDI$ for sample sizes $n=30$ and $n=100$ and shape parameter varying between $0.5$ and $3$}
\begin{tabular}{c|cccc|cccc}
  & \multicolumn{4}{c|}{$30$} & \multicolumn{4}{c}{$100$} \\
\cline{2-9}
 & {0.5} & {1} & {2} & {3} & {0.5} & {1} & {2} & {3} \\
\hline
  HF & 2.108 & 2.687 & 2.643 & 2.113 & 0.617 & 0.828 & 0.798 & 0.642 \\ 
  MDE & 2.522 & 2.981 & 2.662 & 2.060 & 0.741 & 0.895 & 0.780 & 0.607 \\ 
  MDHF & 2.389 & 2.856 & 2.559 & 1.996 & 0.728 & 0.886 & 0.767 & 0.598 \\ 
  ML & 0.596 & 1.319 & 1.748 & 1.478 & 0.150 & 0.346 & 0.481 & 0.425 \\ 
  MML & 0.544 & 1.219 & 1.647 & \color{red}1.431 & 0.146 & 0.330 & 0.469 & \color{red}0.420 \\ 
  BCML & \color{red}0.509 & \color{red}1.171 & \color{red}1.628 & 1.452 & \color{red}0.143 & \color{red}0.326 & \color{red}0.468 & 0.423 \\  
\end{tabular}
\label{tab:mse_qdi}
\end{table}

\begin{figure}[H]
\includegraphics[angle=0,width=\textwidth]
{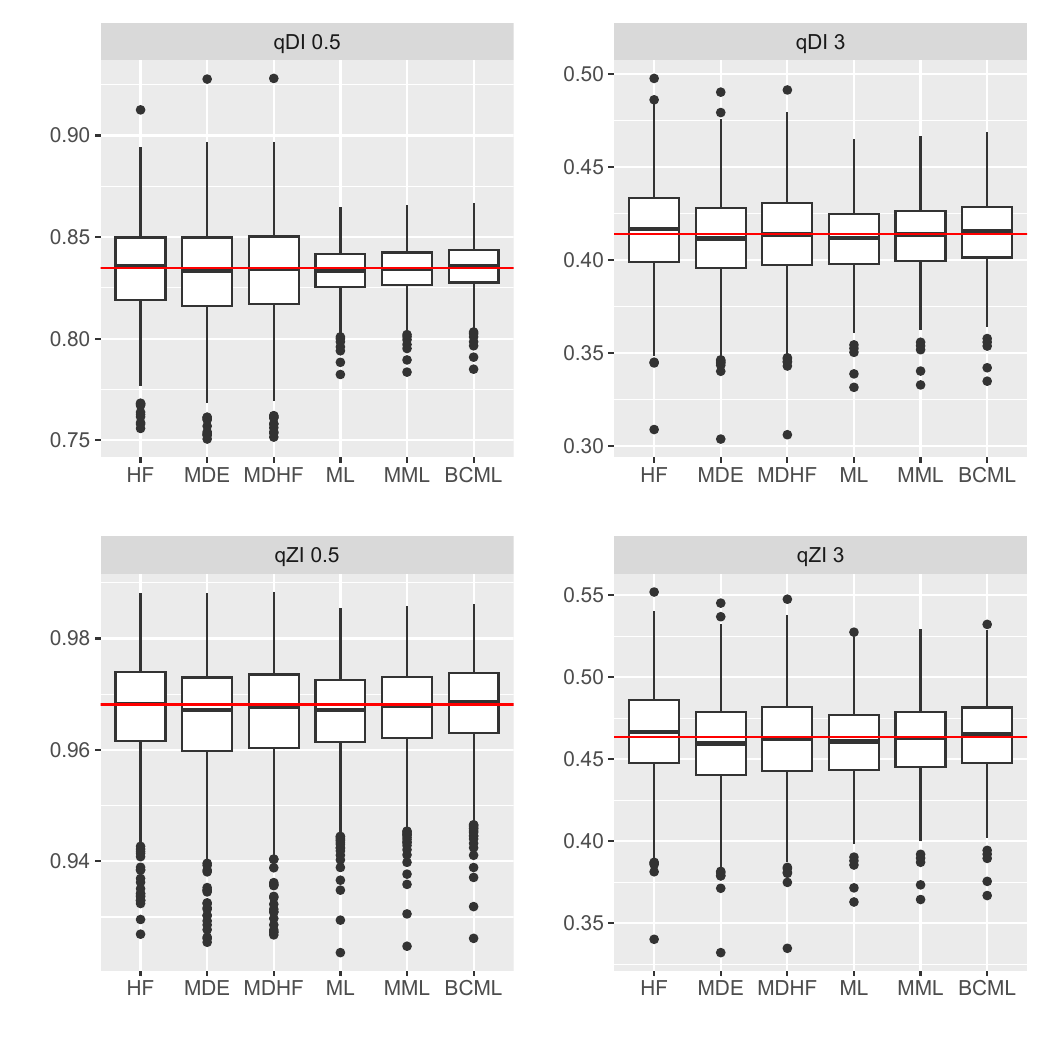}
\caption[Boxplots depicting values of the estimators of $qDI$ (upper row) and $qZI$ (lower row) for sample of size $n=100$ from Weibull distribution with shape parameter $\beta=0.5$ (left column) and $\beta=3$ (right column) compared to the true value given with the red line]{Boxplots depicting values of the estimators of $qDI$ (upper row) and $qZI$ (lower row) for sample of size $n=100$ from Weibull distribution with shape parameter $\beta=0.5$ (left column) and $\beta=3$ (right column) compared to the true value given with the red line}
\label{fig:index_boxplot}
\end{figure}

\section{Real data analysis}\label{sec:real_data_analysis}
We analyse here a~data set containing numbers of days survived by guinea pigs infected with \textit{tubercle bacilli} and a~control group. These data appeared originally in \cite{Bjerkedal1960} and were also considered by Doksum \cite{Doksum1974} and by Chandra and Singpurwalla \cite{Chandra1981}. 
According to the Anderson-Darling test (implemented in \textit{goftest} package in R), there is no reason to reject the hypothesis that survival times of both groups have a~Weibull distribution (with $p$ value $0.7381$ and $0.3702$, respectively). 
The test was performed for a~null hypothesis that the distribution is Weibull with unknown parameters.

The values of the MML estimator of $\beta$ are $1.539$ for the control group and $2.201$ for the treated group.
Figures \ref{fig:gp_qZ} and \ref{fig:gp_qD}  show the $qZ$ and $qD$ curves, respectively, for both groups constructed with the MML estimator (thick smooth lines) together with their empirical counterparts (thin step lines). 
The $x$ axis represents $p$~ being an argument of $qZ$ or $qD$, and the values on the $y$ axis stand for the values of the concentration curves. 
A~comparison of the estimators of the indices for both groups can be found in Table \ref{tab:rda_indices}.

A~Lorenz curve was applied by Chandra and Singpurwalla \cite{Chandra1981} to these data to compare the survival pattern of these two groups. The conclusion was that the infected group had a~less sparse distribution of survival times. 
A~similar conclusion can be drawn from the application of curves $qZ$ and $qD$. Both quantile concentration curves curves show that the control group has a~significantly more sparse survival time.
The empirical Lorenz curves cross near the point of $p=.8$ and on the interval $[0.8,1]$ they are almost equal. 
Empirical $qZ$ and $qD$ curves behave in a different way in the neighbourhood of 1; empirical $qZ$ curves also cross, but close to $p=.9$.
Empirical $qD$ curves do not cross, but approach to each other on $[0.8,1]$.
From this comparison we can see that each of this concentration curves show the differences in the survival pattern of the two groups in a different way, thus each of them might be useful to detect different types of inequalities in the populations.

\begin{minipage}[b]{0.45\textwidth}
\begin{figure}[H]
\includegraphics[angle=0,width=\textwidth]
{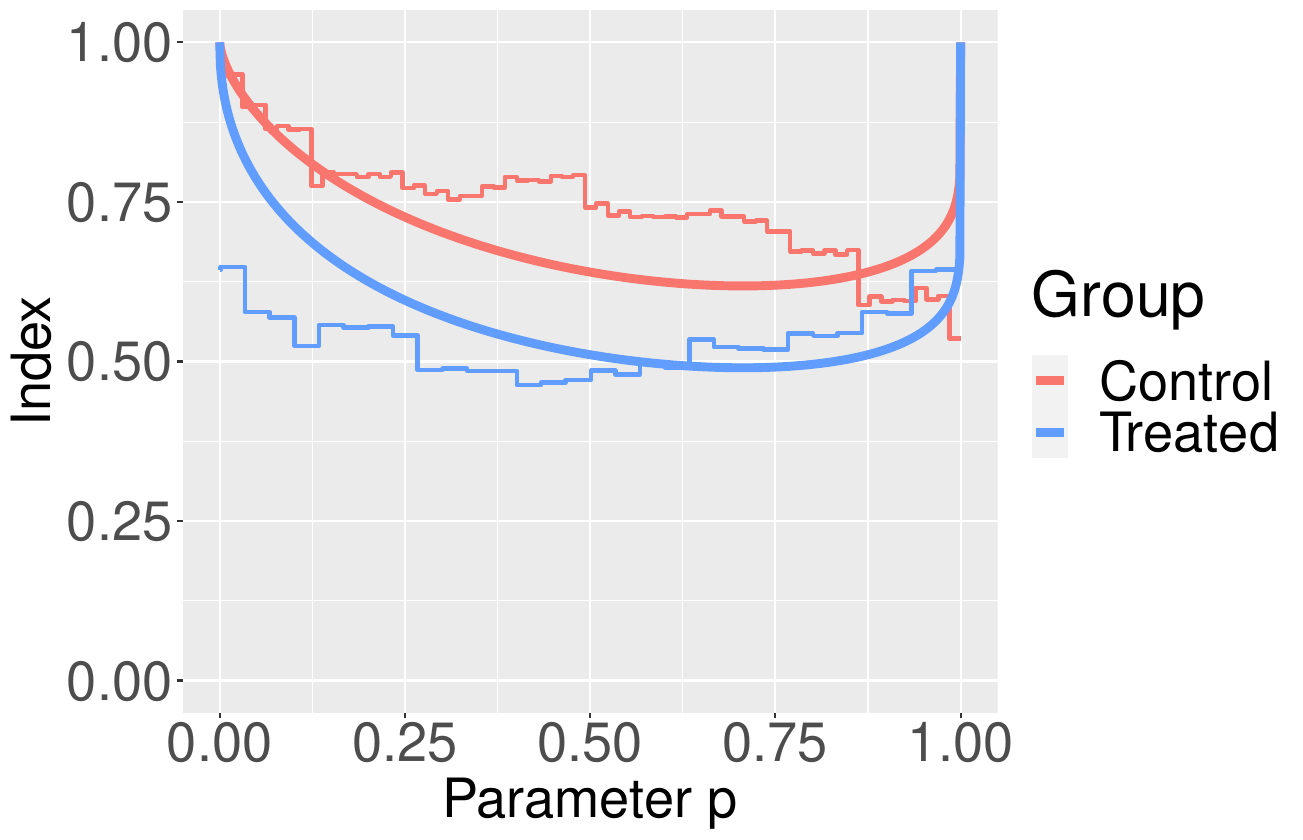}
\caption[The $qZ$ of control group (red) and infected group (blue)]{The $qZ$ of control group (red) and infected group (blue)}
\label{fig:gp_qZ}
\end{figure}
\end{minipage}
\hfill
\begin{minipage}[b]{0.45\textwidth}
\begin{figure}[H]
\includegraphics[angle=0,width=\textwidth]
{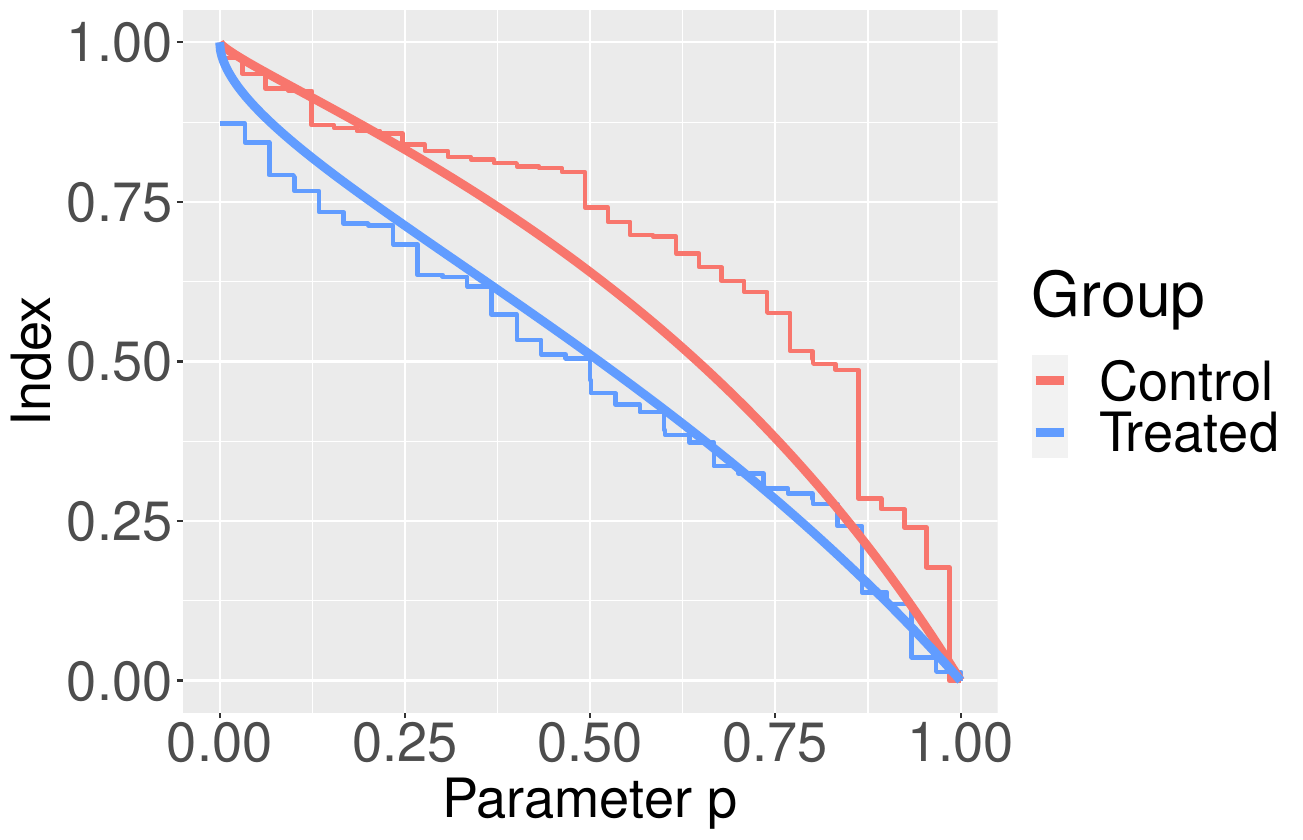}
\caption[The $qD$ of control group (red) and infected group (blue)]{The $qD$ of control group (red) and infected group (blue)}\label{fig:gp_qD}
\end{figure}
\end{minipage}

\begin{table}
\centering
\caption{Estimates of the shape parameter, $qZI$ and $qDI$ obtained with MML estimator}
\begin{tabular}{c|c|c|c}
 & $\hat{\beta}^{MML}$& $\widehat{qZI}^{MML}$ & $\widehat{qDI}^{MML}$  \\
\hline
Control group & 1.539 & 0.6941 & 0.5683  \\
Treated group & 2.201  & 0.5935 & 0.4968  \\
\end{tabular}
\label{tab:rda_indices}
\end{table}

\section{Conclusions}\label{sec:conclusions}
In this article two novel parametric estimators of concentration curves $qZ$ and $qD$, based on the concept of minimum distance (MD) estimation, were introduced. They were compared to some other parametric estimators in case of Weibull distribution. 
It can be concluded that the class of MD estimators performs better than the nonparametric estimators, but worse than the plug-in estimators using several popular estimators of shape parameter. 
The experiment presented in Section \ref{sec:simulation_study} refers to the problem of estimating functions of parameters. 
In this particular case we could see that plug-in estimators based on standard estimators of the parameter give better estimators of the functions of parameters than estimators which are constructed directly to estimate that functions.
This might be the case also in many similar problems when the object of estimation is a~function of a~parameter of a~distribution of a~random variable, not the parameter itself. 

It can also be noticed that the MML estimator of the shape parameter of the Weibull distribution is slightly better than ML in all the cases considered in this paper.
The BCML estimator of both indices has the lowest MSE for small values of $\beta$ (0.5, 1 and 2) and the MML has the lowest MSE for large values of $\beta$ (3).
The advantage of MD estimators is their existence, thus they can be applied in cases when ML or other popular estimators might not exist or might be difficult to compute.
Moreover, they may be recommended in the case when a~good fit to the empirical concentration curve is desired. 
However, according to the simulation study, the curve matching to the empirical curve does not lead to the best estimation, comparing to the true curve.

\vspace{1cm}
\noindent
\textbf{Acknowledgement.} The author would like to thank Alicja Jokiel-Rokita for careful reading of the manuscript and useful remarks and Rafał Topolnicki for valuable comments on the manuscript and on the simulation study.


\newpage
\section{Supplementary material}\label{sec:supplement}

\subsection{Introduction}\label{sec:intro}
In this supplementary material, the additional six methods are used to obtain plug-in estimators of $qZ$ and $qD$ curves and plug-in estimators of $qZI$ and $qDI$ indices. This methods use following estimators of the shape parameter $\beta$ of the Weibull distribution:
\begin{itemize}
    \item Moment estimator (ME),
    \item L-moments estimator (LM),
    \item Tiku's modified maximum likelihood estimator (TMML),
    \item Least squares estimator (LS),
    \item Weighted least squares estimator (WLS),
    \item Gini index-based estimator (G1).
\end{itemize}
Detailed description of ME, LM, TMML, LS and WLS can be found in \cite{Gebizlioglu2011}, while G1 estimator is described in Section 3.2 in \cite{JokPia2022} as $\hat{\beta}_{\hat{G}}$.
Moreover, an additional sample size 500 is considered and tables comparing bias of estimators of indices are included.

The estimators MML and BCML were chosen for a~comparison in the main part of the work, since they were performing best (in terms of MSE and MISE) among the plug-in estimators. ML was added to the comparison because it is the state-of-the-art method used for estimation in variety of problems.

Analogously as in Section 5 of the main work, four values of the shape parameter $\beta$ were considered, namely $\beta\in\{0.5,1,2,3\}$.
For each case, 10,000 samples were generated. Mean integrated square error (MISE) was used as a~criterion to assess the estimators of the curves. The estimators of indices were compared in terms of MSE and bias.

\subsection{Results of the extended simulation study}
Twelve following tables contain comparison of performance of estimators of $qZ$, $qD$, $qZI$ and $qDI$. 
Tables \ref{tab:mseqz}, \ref{tab:mseqd}, \ref{tab:mseqz500} and \ref{tab:mseqd500} contain comparison of MSE of estimators.
Tables \ref{tab:miseqz}, \ref{tab:miseqd}, \ref{tab:miseqz500} and \ref{tab:miseqd500} contain comparison of MISE of estimators.
Tables \ref{tab:biasqz}, \ref{tab:biasqd}, \ref{tab:biasqz500} and \ref{tab:biasqd500} contain comparison of bias of estimators. 
Tables \ref{tab:mseqz}--\ref{tab:biasqd} contain comparison of estimators in case of sample sizes ($n$) 30 and 100, while \ref{tab:mseqz500}--\ref{tab:biasqd500} refer to the case when $n=500$.

\begin{table}[H]
\centering
\caption{MSE (multiplied by 1,000) of the estimators of $qZI$ for sample sizes $n=30$ and $n=100$ and shape parameter varying between $0.5$ and $3$} 
\label{tab:mseqz}
\begin{tabular}{c|cccc|cccc}
	& \multicolumn{4}{c|}{$30$} & \multicolumn{4}{c}{$100$} \\
	\cline{2-9}
  \hline
 & 0.5 & 1 & 2 & 3 & 0.5 & 1 & 2 & 3 \\ 
  \hline
HF & 0.540 & 2.444 & 3.362 & 2.727 & 0.112 & 0.675 & 0.992 & 0.818 \\ 
  MDE & 0.631 & 2.714 & 3.613 & 2.927 & 0.123 & 0.722 & 1.025 & 0.847 \\ 
  MDHF & 0.544 & 2.466 & 3.408 & 2.805 & 0.117 & 0.694 & 0.998 & 0.827 \\ 
  ML & 0.371 & 2.046 & 2.896 & 2.260 & 0.077 & 0.523 & 0.806 & 0.656 \\ 
  MML & 0.318 & 1.866 & 2.742 & 2.203 & 0.073 & 0.500 & 0.787 & 0.650 \\ 
  BCML & 0.270 & 1.753 & 2.725 & 2.257 & 0.070 & 0.489 & 0.786 & 0.656 \\ 
  ME & 1.472 & 2.497 & 2.725 & 2.245 & 0.395 & 0.760 & 0.792 & 0.667 \\ 
  LM & 0.513 & 2.007 & 2.839 & 2.384 & 0.136 & 0.563 & 0.823 & 0.703 \\ 
  TMML & 0.367 & 2.022 & 2.865 & 2.264 & 0.078 & 0.517 & 0.799 & 0.656 \\ 
  LS & 0.448 & 2.571 & 3.954 & 3.382 & 0.105 & 0.737 & 1.157 & 0.970 \\ 
  WLS & 0.394 & 2.257 & 3.375 & 2.857 & 0.088 & 0.605 & 0.944 & 0.785 \\ 
  G1 & 0.776 & 2.372 & 3.023 & 2.449 & 0.159 & 0.598 & 0.839 & 0.708 \\ 
\end{tabular}
\end{table}

\begin{table}[H]
\centering
\caption{MSE (multiplied by 1,000) of the estimators of $qDI$ for sample sizes $n=30$ and $n=100$ and shape parameter varying between $0.5$ and $3$} 
\label{tab:mseqd}
\begin{tabular}{c|cccc|cccc}
	& \multicolumn{4}{c|}{$30$} & \multicolumn{4}{c}{$100$} \\
	\cline{2-9}
	\hline
	& 0.5 & 1 & 2 & 3 & 0.5 & 1 & 2 & 3 \\ 
	\hline
HF & 2.108 & 2.687 & 2.643 & 2.113 & 0.617 & 0.828 & 0.798 & 0.642 \\ 
  MDE & 2.522 & 2.981 & 2.662 & 2.060 & 0.741 & 0.895 & 0.780 & 0.607 \\ 
  MDHF & 2.389 & 2.856 & 2.559 & 1.996 & 0.728 & 0.886 & 0.767 & 0.598 \\ 
  ML & 0.596 & 1.319 & 1.748 & 1.478 & 0.150 & 0.346 & 0.481 & 0.425 \\ 
  MML & 0.544 & 1.219 & 1.647 & 1.431 & 0.146 & 0.330 & 0.469 & 0.420 \\ 
  BCML & 0.509 & 1.171 & 1.628 & 1.452 & 0.143 & 0.326 & 0.468 & 0.423 \\ 
  ME & 1.905 & 1.612 & 1.636 & 1.458 & 0.638 & 0.502 & 0.472 & 0.431 \\ 
  LM & 0.865 & 1.329 & 1.701 & 1.543 & 0.267 & 0.374 & 0.490 & 0.454 \\ 
  TMML & 0.594 & 1.304 & 1.727 & 1.479 & 0.150 & 0.340 & 0.477 & 0.425 \\ 
  LS & 0.788 & 1.726 & 2.366 & 2.175 & 0.213 & 0.491 & 0.688 & 0.625 \\ 
  WLS & 0.681 & 1.499 & 2.022 & 1.845 & 0.175 & 0.401 & 0.562 & 0.507 \\ 
  G1 & 1.106 & 1.516 & 1.825 & 1.605 & 0.288 & 0.392 & 0.501 & 0.459 \\ 
\end{tabular}
\end{table}

\begin{table}[H]
\centering
\caption{MISE (multiplied by 1,000) of the estimators of $qZ$ for sample sizes $n=30$ and $n=100$ and shape parameter varying between $0.5$ and $3$} 
\label{tab:miseqz}
\begin{tabular}{c|cccc|cccc}
	& \multicolumn{4}{c|}{$30$} & \multicolumn{4}{c}{$100$} \\
	\cline{2-9}
	\hline
	& 0.5 & 1 & 2 & 3 & 0.5 & 1 & 2 & 3 \\ 
	\hline
HF & 0.861 & 4.399 & 7.155 & 6.705 & 0.215 & 1.358 & 2.204 & 2.086 \\ 
  MDE & 0.735 & 2.851 & 3.633 & 2.946 & 0.147 & 0.764 & 1.031 & 0.852 \\ 
  MDHF & 0.635 & 2.596 & 3.429 & 2.822 & 0.140 & 0.735 & 1.005 & 0.832 \\ 
  ML & 0.435 & 2.152 & 2.912 & 2.274 & 0.093 & 0.554 & 0.811 & 0.659 \\ 
  MML & 0.375 & 1.968 & 2.759 & 2.216 & 0.088 & 0.530 & 0.792 & 0.654 \\ 
  BCML & 0.321 & 1.857 & 2.745 & 2.268 & 0.084 & 0.519 & 0.791 & 0.660 \\ 
  ME & 1.696 & 2.625 & 2.742 & 2.259 & 0.465 & 0.805 & 0.797 & 0.670 \\ 
  LM & 0.603 & 2.122 & 2.858 & 2.397 & 0.162 & 0.597 & 0.828 & 0.707 \\ 
  TMML & 0.431 & 2.127 & 2.881 & 2.278 & 0.093 & 0.547 & 0.804 & 0.659 \\ 
  LS & 0.526 & 2.722 & 3.985 & 3.400 & 0.126 & 0.782 & 1.164 & 0.975 \\ 
  WLS & 0.463 & 2.386 & 3.399 & 2.873 & 0.106 & 0.642 & 0.950 & 0.789 \\ 
  G1 & 0.901 & 2.490 & 3.040 & 2.465 & 0.189 & 0.633 & 0.844 & 0.712 \\ 
\end{tabular}
\end{table}

\begin{table}[H]
\centering
\caption{MISE (multiplied by 1,000) of the estimators of $qD$ for sample sizes $n=30$ and $n=100$ and shape parameter varying between $0.5$ and $3$} 
\label{tab:miseqd}
\begin{tabular}{c|cccc|cccc}
	& \multicolumn{4}{c|}{$30$} & \multicolumn{4}{c}{$100$} \\
	\cline{2-9}
	\hline
	& 0.5 & 1 & 2 & 3 & 0.5 & 1 & 2 & 3 \\ 
	\hline
HF & 5.914 & 5.483 & 4.834 & 4.261 & 2.033 & 1.853 & 1.534 & 1.320 \\ 
  MDE & 4.192 & 3.675 & 3.018 & 2.408 & 1.286 & 1.108 & 0.882 & 0.707 \\ 
  MDHF & 4.037 & 3.546 & 2.900 & 2.325 & 1.269 & 1.099 & 0.867 & 0.696 \\ 
  ML & 0.989 & 1.604 & 1.982 & 1.732 & 0.259 & 0.428 & 0.544 & 0.495 \\ 
  MML & 0.918 & 1.492 & 1.865 & 1.671 & 0.253 & 0.408 & 0.530 & 0.489 \\ 
  BCML & 0.880 & 1.450 & 1.843 & 1.688 & 0.251 & 0.404 & 0.529 & 0.492 \\ 
  ME & 2.960 & 1.960 & 1.853 & 1.703 & 1.057 & 0.619 & 0.533 & 0.502 \\ 
  LM & 1.463 & 1.637 & 1.926 & 1.798 & 0.463 & 0.463 & 0.554 & 0.528 \\ 
  TMML & 0.986 & 1.585 & 1.957 & 1.732 & 0.260 & 0.419 & 0.539 & 0.495 \\ 
  LS & 1.349 & 2.138 & 2.680 & 2.527 & 0.372 & 0.609 & 0.778 & 0.726 \\ 
  WLS & 1.158 & 1.848 & 2.291 & 2.148 & 0.305 & 0.496 & 0.635 & 0.590 \\ 
  G1 & 1.774 & 1.833 & 2.069 & 1.883 & 0.490 & 0.482 & 0.566 & 0.535 \\ 
\end{tabular}
\end{table}

\begin{table}[H]
\centering
\caption{Bias (multiplied by 1,000) of the estimators of $qZ$ for sample sizes $n=30$ and $n=100$ and shape parameter varying between $0.5$ and $3$} 
\label{tab:biasqz}
\begin{tabular}{c|cccc|cccc}
	& \multicolumn{4}{c|}{$30$} & \multicolumn{4}{c}{$100$} \\
	\cline{2-9}
	\hline
	& 0.5 & 1 & 2 & 3 & 0.5 & 1 & 2 & 3 \\ 
	\hline
HF & -8.562 & -12.138 & -9.065 & -5.756 & -2.458 & -3.766 & -2.198 & -1.177 \\ 
  MDE & -9.912 & -15.680 & -14.410 & -11.931 & -2.860 & -5.160 & -4.407 & -3.776 \\ 
  MDHF & -7.702 & -10.462 & -7.921 & -5.786 & -2.297 & -3.372 & -1.893 & -1.104 \\ 
  ML & -6.843 & -12.915 & -13.290 & -10.044 & -1.845 & -3.995 & -3.717 & -2.727 \\ 
  MML & -4.488 & -6.967 & -6.030 & -3.370 & -1.191 & -2.309 & -1.576 & -0.765 \\ 
  BCML & -1.685 & 0.399 & 3.108 & 5.060 & -0.381 & -0.109 & 1.151 & 1.746 \\ 
  ME & -26.454 & -16.530 & -6.218 & -3.018 & -11.053 & -5.759 & -1.701 & -0.651 \\ 
  LM & -6.675 & -3.970 & -1.516 & -0.037 & -2.103 & -1.490 & -0.280 & 0.157 \\ 
  TMML & -6.666 & -12.448 & -12.722 & -9.500 & -1.867 & -4.119 & -3.788 & -2.831 \\ 
  LS & -2.259 & 1.557 & 6.387 & 8.042 & -0.583 & 0.249 & 1.913 & 2.448 \\ 
  WLS & -2.888 & -1.235 & 2.029 & 3.867 & -0.994 & -1.385 & -0.129 & 0.371 \\ 
  G1 & -14.936 & -17.861 & -15.634 & -12.233 & -4.420 & -5.615 & -4.505 & -3.494 \\ 
\end{tabular}
\end{table}

\begin{table}[H]
\centering
\caption{Bias (multiplied by 1,000) of the estimators of $qD$ for sample sizes $n=30$ and $n=100$ and shape parameter varying between $0.5$ and $3$} 
\label{tab:biasqd}
\begin{tabular}{c|cccc|cccc}
	& \multicolumn{4}{c|}{$30$} & \multicolumn{4}{c}{$100$} \\
	\cline{2-9}
	\hline
	& 0.5 & 1 & 2 & 3 & 0.5 & 1 & 2 & 3 \\ 
	\hline
HF & -10.529 & -7.425 & -4.419 & -2.799 & -3.091 & -2.642 & -1.246 & -0.595 \\ 
  MDE & -12.732 & -9.318 & -7.384 & -6.731 & -3.341 & -3.150 & -2.264 & -2.069 \\ 
  MDHF & -9.495 & -5.470 & -2.367 & -1.438 & -2.609 & -2.045 & -0.715 & -0.178 \\ 
  ML & -7.141 & -9.750 & -10.483 & -8.415 & -1.936 & -3.024 & -2.926 & -2.289 \\ 
  MML & -3.911 & -4.918 & -4.867 & -3.040 & -1.000 & -1.661 & -1.275 & -0.712 \\ 
  BCML & 0.062 & 1.121 & 2.196 & 3.731 & 0.172 & 0.135 & 0.829 & 1.304 \\ 
  ME & -30.234 & -12.525 & -5.010 & -2.763 & -13.061 & -4.363 & -1.372 & -0.623 \\ 
  LM & -5.582 & -2.366 & -1.385 & -0.387 & -1.726 & -0.960 & -0.278 & 0.020 \\ 
  TMML & -6.890 & -9.378 & -10.048 & -7.978 & -1.967 & -3.134 & -2.982 & -2.373 \\ 
  LS & 0.521 & 2.444 & 4.645 & 5.964 & 0.226 & 0.542 & 1.389 & 1.822 \\ 
  WLS & -0.922 & -0.020 & 1.313 & 2.682 & -0.565 & -0.858 & -0.171 & 0.180 \\ 
  G1 & -16.736 & -13.705 & -12.311 & -10.204 & -5.020 & -4.323 & -3.538 & -2.914 \\ 
\end{tabular}
\end{table}

\begin{table}[H]
\centering
\caption{MSE (multiplied by 10,000) of the estimators of $qZI$ for sample size $n=500$ and shape parameter varying between $0.5$ and $3$} 
\label{tab:mseqz500}
\begin{tabular}{c|cccc}
  \hline
 & 0.5 & 1 & 2 & 3 \\ 
  \hline
HF & 0.1938 & 1.3409 & 1.9393 & 1.6931 \\ 
  MDE & 0.2093 & 1.3938 & 1.9611 & 1.7199 \\ 
  MDHF & 0.2066 & 1.3825 & 1.9509 & 1.7059 \\ 
  ML & 0.1332 & 1.0089 & 1.5247 & 1.3087 \\ 
  MML & 0.1309 & 1.0046 & 1.5210 & 1.3027 \\ 
  BCML & 0.1293 & 1.0027 & 1.5217 & 1.3029 \\ 
  ME & 0.8594 & 1.5746 & 1.5514 & 1.3461 \\ 
  LM & 0.2509 & 1.1305 & 1.5972 & 1.4212 \\ 
  TMML & 0.1330 & 1.0111 & 1.5268 & 1.3075 \\ 
  LS & 0.1991 & 1.4873 & 2.2497 & 1.9713 \\ 
  WLS & 0.1583 & 1.1866 & 1.7861 & 1.5540 \\ 
  G1 & 0.2606 & 1.1380 & 1.6014 & 1.4253 \\ 
\end{tabular}
\end{table}

\begin{table}[H]
\centering
\caption{MSE (multiplied by 10,000) of the estimators of $qDI$ for sample size $n=500$ and shape parameter varying between $0.5$ and $3$} 
\label{tab:mseqd500}
\begin{tabular}{c|cccc}
  \hline
 & 0.5 & 1 & 2 & 3 \\ 
  \hline
HF & 1.2272 & 1.6897 & 1.5778 & 1.3406 \\ 
  MDE & 1.4756 & 1.8180 & 1.5194 & 1.2435 \\ 
  MDHF & 1.4680 & 1.8162 & 1.5159 & 1.2372 \\ 
  ML & 0.2774 & 0.6698 & 0.9067 & 0.8453 \\ 
  MML & 0.2737 & 0.6677 & 0.9043 & 0.8411 \\ 
  BCML & 0.2725 & 0.6674 & 0.9044 & 0.8408 \\ 
  ME & 1.7127 & 1.0462 & 0.9223 & 0.8691 \\ 
  LM & 0.5233 & 0.7521 & 0.9494 & 0.9173 \\ 
  TMML & 0.2764 & 0.6711 & 0.9080 & 0.8447 \\ 
  LS & 0.4179 & 0.9904 & 1.3371 & 1.2714 \\ 
  WLS & 0.3306 & 0.7887 & 1.0618 & 1.0031 \\ 
  G1 & 0.5332 & 0.7551 & 0.9524 & 0.9207 \\ 
\end{tabular}
\end{table}

\begin{table}[H]
\centering
\caption{MISE (multiplied by 10,000) of the estimators of $qZ$ for sample size $n=500$ and shape parameter varying between $0.5$ and $3$} 
\label{tab:miseqz500}
\begin{tabular}{c|cccc}
  \hline
 & 0.5 & 1 & 2 & 3 \\ 
  \hline
HF & 0.4028 & 2.7817 & 4.4926 & 4.2785 \\ 
  MDE & 0.2526 & 1.4802 & 1.9732 & 1.7284 \\ 
  MDHF & 0.2495 & 1.4686 & 1.9631 & 1.7143 \\ 
  ML & 0.1609 & 1.0716 & 1.5340 & 1.3152 \\ 
  MML & 0.1582 & 1.0673 & 1.5305 & 1.3091 \\ 
  BCML & 0.1564 & 1.0656 & 1.5312 & 1.3093 \\ 
  ME & 1.0323 & 1.6726 & 1.5610 & 1.3527 \\ 
  LM & 0.3030 & 1.2013 & 1.6072 & 1.4282 \\ 
  TMML & 0.1607 & 1.0740 & 1.5361 & 1.3140 \\ 
  LS & 0.2407 & 1.5807 & 2.2639 & 1.9809 \\ 
  WLS & 0.1913 & 1.2607 & 1.7972 & 1.5616 \\ 
  G1 & 0.3143 & 1.2086 & 1.6113 & 1.4324 \\ 
\end{tabular}
\end{table}

\begin{table}[H]
\centering
\caption{MISE (multiplied by 10,000) of the estimators of $qD$ for sample size $n=500$ and shape parameter varying between $0.5$ and $3$} 
\label{tab:miseqd500}
\begin{tabular}{c|cccc}
  \hline
 & 0.5 & 1 & 2 & 3 \\ 
  \hline
HF & 4.2987 & 3.8925 & 3.1423 & 2.7204 \\ 
  MDE & 2.5884 & 2.2570 & 1.7164 & 1.4467 \\ 
  MDHF & 2.5746 & 2.2554 & 1.7124 & 1.4390 \\ 
  ML & 0.4858 & 0.8303 & 1.0242 & 0.9837 \\ 
  MML & 0.4798 & 0.8282 & 1.0214 & 0.9786 \\ 
  BCML & 0.4783 & 0.8285 & 1.0215 & 0.9779 \\ 
  ME & 2.9800 & 1.2975 & 1.0418 & 1.0111 \\ 
  LM & 0.9171 & 0.9334 & 1.0724 & 1.0671 \\ 
  TMML & 0.4838 & 0.8319 & 1.0256 & 0.9830 \\ 
  LS & 0.7334 & 1.2295 & 1.5103 & 1.4784 \\ 
  WLS & 0.5794 & 0.9783 & 1.1994 & 1.1668 \\ 
  G1 & 0.9305 & 0.9358 & 1.0758 & 1.0714 \\ 
\end{tabular}
\end{table}

\begin{table}[H]
\centering
\caption{Bias (multiplied by 1,000) of the estimators of $qZI$ for sample size $n=500$ and shape parameter varying between $0.5$ and $3$} 
\label{tab:biasqz500}
\begin{tabular}{c|cccc}
  \hline
 & 0.5 & 1 & 2 & 3 \\ 
  \hline
HF & -0.488 & -0.431 & -0.159 & -0.250 \\ 
  MDE & -0.577 & -0.751 & -0.747 & -1.017 \\ 
  MDHF & -0.457 & -0.363 & -0.130 & -0.297 \\ 
  ML & -0.393 & -0.517 & -0.668 & -0.752 \\ 
  MML & -0.269 & -0.170 & -0.239 & -0.370 \\ 
  BCML & -0.109 & 0.269 & 0.306 & 0.130 \\ 
  ME & -3.307 & -0.790 & -0.254 & -0.314 \\ 
  LM & -0.498 & 0.100 & 0.065 & -0.122 \\ 
  TMML & -0.422 & -0.582 & -0.747 & -0.835 \\ 
  LS & -0.107 & 0.369 & 0.607 & 0.428 \\ 
  WLS & -0.260 & -0.099 & -0.049 & -0.192 \\ 
  G1 & -0.950 & -0.721 & -0.779 & -0.851 \\ 
\end{tabular}
\end{table}

\begin{table}[H]
\centering
\caption{Bias (multiplied by 1,000) of the estimators of $qDI$ for sample size $n=500$ and shape parameter varying between $0.5$ and $3$} 
\label{tab:biasqd500}
\begin{tabular}{c|cccc}
  \hline
 & 0.5 & 1 & 2 & 3 \\ 
  \hline
HF & -0.649 & -0.221 & 0.016 & -0.129 \\ 
  MDE & -0.563 & -0.250 & -0.202 & -0.549 \\ 
  MDHF & -0.619 & -0.116 & 0.108 & -0.095 \\ 
  ML & -0.432 & -0.375 & -0.526 & -0.624 \\ 
  MML & -0.253 & -0.092 & -0.195 & -0.317 \\ 
  BCML & -0.019 & 0.266 & 0.225 & 0.085 \\ 
  ME & -3.954 & -0.571 & -0.207 & -0.272 \\ 
  LM & -0.458 & 0.134 & 0.038 & -0.119 \\ 
  TMML & -0.475 & -0.428 & -0.587 & -0.690 \\ 
  LS & 0.057 & 0.370 & 0.451 & 0.314 \\ 
  WLS & -0.210 & -0.026 & -0.051 & -0.178 \\ 
  G1 & -1.113 & -0.536 & -0.613 & -0.705 \\ 
\end{tabular}
\end{table}

\subsection{Conclusions}
The methods which were not included in the main article perform worse than BCML and MML in all cases considered. The TMML method gives results close to ML, since it is an approximation of ML. In terms of bias it seems that the BCML is best for small $\beta$ (0.5 and 1), while LM is best in case of larger $\beta$ (2 and 3). The simulations performed for larger sample size (500) give similar conclusions as the ones with smaller sample sizes (30 and 100), but in this case MML is slightly better than BCML in terms of MSE and MISE when $\beta$ is 1 and 2.

\end{document}